\documentclass[12pt]{amsart}
\usepackage[nobysame,abbrev,alphabetic]{amsrefs}
\usepackage{amssymb}
\usepackage[all]{xy}

\DeclareMathOperator{\Br}{Br}

\DeclareMathOperator{\Cor}{Cor}
\DeclareMathOperator{\Char}{char}

\DeclareMathOperator{\Gal}{Gal}
\DeclareMathOperator{\Hom}{Hom}

\DeclareMathOperator{\Img}{Im}
\DeclareMathOperator{\Inf}{Inf}

\DeclareMathOperator{\Ker}{Ker}

\DeclareMathOperator{\Res}{Res}

\begin{document}

\newtheorem{thm}{Theorem}[section]
\newtheorem{cor}[thm]{Corollary}
\newtheorem{lem}[thm]{Lemma}
\newtheorem{prop}[thm]{Proposition}
\newtheorem{defin}[thm]{Definition}
\newtheorem{exam}[thm]{Example}
\newtheorem{examples}[thm]{Examples}
\newtheorem{rem}[thm]{Remark}
\newtheorem*{thmA}{Theorem A}
\newtheorem*{mainthm}{Main Theorem}
\swapnumbers
\newtheorem{rems}[thm]{Remarks}
\newtheorem*{acknowledgment}{Acknowledgment}
\numberwithin{equation}{section}

\newcommand{\dec}{\mathrm{dec}}
\newcommand{\dirlim}{\varinjlim}
\newcommand{\discup}{\mathbin{\mathaccent\cdot\cup}}
\newcommand{\nek}{,\ldots,}
\newcommand{\inv}{^{-1}}
\newcommand{\Inv}{\mathrm{inv}}
\newcommand{\isom}{\cong}
\newcommand{\Massey}{\mathrm{Massey}}
\newcommand{\ndiv}{\hbox{$\,\not|\,$}}
\newcommand{\pr}{\mathrm{pr}}
\newcommand{\sep}{{\rm sep}}
\newcommand{\tensor}{\otimes}
\newcommand{\U}{\mathrm{U}}
\newcommand{\alp}{\alpha}
\newcommand{\gam}{\gamma}
\newcommand{\del}{\delta}
\newcommand{\eps}{\epsilon}
\newcommand{\lam}{\lambda}
\newcommand{\Lam}{\Lambda}
\newcommand{\sig}{\sigma}
\newcommand{\dbF}{\mathbb{F}}
\newcommand{\dbQ}{\mathbb{Q}}
\newcommand{\dbR}{\mathbb{R}}
\newcommand{\dbU}{\mathbb{U}}
\newcommand{\dbZ}{\mathbb{Z}}

\title{Triple Massey products and absolute Galois groups}

\author{Ido Efrat and Eliyahu Matzri}

\address{Department of Mathematics \\
         Ben-Gurion University of the Negev\\
         Be'er-Sheva 84105 \\
         Israel}
\email{efrat@math.bgu.ac.il, elimatzri@gmail.com}
\thanks{$^1$The authors were supported by the Israel Science Foundation (grant No.\ 152/13).
The second author was also partially supported by the Kreitman foundation}



\maketitle

\begin{abstract}
Let $p$ be a prime number, $F$ a field containing a root of unity of order $p$, and $G_F$ the absolute Galois group.
Extending results of Hopkins, Wickelgren, Min\'a\v c and T\^an, we prove that the triple Massey product $H^1(G_F)^3\to H^2(G_F)$ contains $0$ whenever it is nonempty.
This gives a new restriction on the possible profinite group structure of $G_F$.
\end{abstract}

A main problem in modern Galois theory is to understand the group-theoretic structure of absolute Galois groups $G_F=\Gal(F_\sep/F)$ of fields $F$, that is, the possible symmetry patterns of roots of polynomials.
General restrictions on the possible structure of the profinite group $G_F$ are rare:
By classical results of Artin and Schreier, the torsion in $G_F$ can consist only of involutions.
In addition, the celebrated work of Voevodsky and Rost  (\cite{Voevodsky02}, \cite{Voevodsky11}) identifies the cohomology ring $H^*(G_F)=H^*(G_F,\dbZ/m)$ with the mod-$m$ Milnor $K$-ring $K^M_*(F)/m$, assuming existence of $m$-th roots of unity.
In particular, the graded ring $H^*(G_F)$ is generated by its degree $1$ elements, and its relations originate from the degree $2$ component.
This can be used to rule out many more profinite groups from being absolute Galois groups of fields (\cite{CheboluEfratMinac12}, \cite{EfratMinac11b}).
In fact, the Artin--Schreier restriction about the torsion also follows from the latter results \cite{EfratMinac11b}*{Ex.\ 6.4(2)}.

Very recently, a remarkable series of works by Hopkins,  Wickelgren, Min\'a\v c and T\^an indicated the possible existence of a new kind of general restrictions on the structure of absolute Galois groups, related to the differential graded algebra $C^*(G_F)=C^*(G_F,\dbZ/m)$ of continuous cochains on $G_F$.
The interplay between $C^*(G_F)$ and its cohomology algebra $H^*(G_F)$ gives rise to \textsl{external}  operations on $H^*(G_F)$, in addition to its (``internal") ring structure with respect to the cup product, notably, the \textsl{$n$-fold Massey products} $H^1(G_F)^n\to H^2(G_F)$.
The definition of the Massey product in the context of general differential algebras is recalled in \S\ref{section on Massey products}, and  at this stage we only mention that it is a multi-valued map, which for $n=2$ coincides with the cup product.
The Massey product $\langle\chi_1\nek\chi_n\rangle\subseteq H^2(G_F)$ is \textsl{essential} if it is non-empty, but does not contain $0$.
The above-mentioned works show that, under various assumptions, the \textsl{triple} Massey product for $H^*(G_F)$ is never essential.
Thus profinite groups $G$ for which $H^*(G)$ contains an essential triple Massey product cannot be realized as absolute Galois groups of fields satisfying these assumptions.
In \cite{MinacTan14a} Min\'a\v c and T\^an develop a method to produce such groups $G$, by examining their presentation by generators and relations modulo the $4$th term in the $p$-Zassenhaus filtration.
As a concrete example, the profinite group $G$ on $5$ generators $\sig_1\nek\sig_5$ and the single defining relation $[\sig_4,\sig_5][[\sig_2,\sig_3],\sig_1]$ gives rise to an essential triple Massey product \cite{MinacTan14a}*{Ex.\ 7.2}.

Specifically, assume that $m=p$ is prime, and $F$ contains a root of unity of order $p$ (so $\Char\,F\neq p$).
It was shown that the triple Massey product for $H^*(G_F)$ is never essential in the following situations:
\begin{enumerate}
\item[1)]
$p=2$ and $F$ is a local field or a global field (Hopkins and Wickelgren \cite{HopkinsWickelgren15});
\item[2)]
$p=2$ and $F$ is arbitrary (Min\'a\v c and T\^an \cite{MinacTan14a});
\item[3)]
$p$ is arbitrary and $F$ is a local field (Min\'a\v c and T\^an; follows from \cite{MinacTan14a}*{Th.\ 4.3} and \cite{MinacTan13}*{Th.\ 8.5});
\item[4)]
$p$ is arbitrary, and $F$ is a global field  (Min\'a\v c and T\^an \cite{MinacTan14b}).
\end{enumerate}
Moreover, it is conjectured in \cite{MinacTan13} that the $n$-fold Massey product above is never essential for every $n\geq3$.
Also, in \cite{EfratMatzri15} we find close connections between these results and classical facts in the theory of central simple algebras.
In particular, 2) is closely related to Albert's characterization from 1939 \cite{Albert39}  (as refined by Tignol \cite{Tignol79}; see also Rowen \cite{Rowen84} and \cite{Tignol81}) of the central simple algebras of exponent $2$ and degree $4$ as biquaternionic algebras.

Motivated by these works, we prove in this paper the above conjecture for triple Massey products for arbitrary $p$ and general fields $F$ as above:

\begin{mainthm}
Let $F$ be a field containing a root of unity of order $p$, and let $\chi_1,\chi_2,\chi_3\in H^1(G_F)$.
Then $\langle\chi_1,\chi_2,\chi_3\rangle$ is not essential.
\end{mainthm}

The Main Theorem was first proved by the second-named author using methods from the theory of central simple algebras, notably the Amitsur--Saltman theory of abelian crossed products \cite{Matzri14}.
The current paper, which replaces \cite{Matzri14}, is based on a shortcut which allows carrying the original crossed product computations to the framework of profinite group cohomology (see Proposition \ref{key prop}).
We also work in a more general formal context, and prove the Main Theorem for \textsl{$p$-Kummer formations} $(G,A,\{\kappa_U\}_U)$ (Theorem \ref{Main Theorem for formations}).
These structures axiomatize the relevant Galois-theoretic properties of absolute Galois groups:
the Kummer isomorphism, Hilbert's Theorem 90, and the connections between restriction, correstriction, and cup product.
The Main Theorem is just the case where $G=G_F$, $A=F_{\sep}^\times$, and the $\kappa_U$ are the Kummer maps
(see \S\ref{section on Kummer formations}).

The Main Theorem is in a partial analogy with the important work of Deligne, Griffiths, Morgan, and Sullivan \cite{DeligneGriffithsMorganSullivan75}, which proves that any compact K\"ahler manifold is formal.
This implies that its $n$-fold Massey products, with $n\geq3$, are non-essential in the de Rahm context (see also \cite{Huybrechts05}*{Ch.\ 3.A}).
On the other hand, links  in $\dbR^3$ provide examples of essential  Massey products in the algebra of singular cochains.
For instance, the \textsl{Borromean rings} give rise to an essential triple Massey product  \cite{Hillman12}*{\S10.1}, and this explains why they are not equivalent to three unconnected circles.
Thus the Main Theorem means that a phenomena such as the Borromean rings is impossible in this Galois cohomology context.
We also note that examples due to Positselski show that $H^*(G_F)$ may not be formal (\cite{Positselski11}*{\S9.11}, \cite{Positselski15}).

Among the other works on Massey products in Galois cohomology we mention those by Morishita \cite{Morishita04},  Sharifi (\cite{SharifiThesis}, \cite{Sharifi07}), Wickelgren (\cite{Wickelgren12a}, \cite{Wickelgren12b}), Vogel \cite{Vogel05}, G\"artner \cite{Gartner15}, and the first-named author \cite{Efrat14}.

We thank J\'an Min\'a\v c, Leonid Positselski, Louis Rowen, Nguyen Duy T\^an, Uzi Vishne and Kirsten Wickelgren for discussions over the past few years on various aspects of Massey products and of this work.
We also thank the referee for his/her very valuable comments.

\bigskip

{\sl Addendum (January 2015):} \quad
In the recent paper \cite{MinacTan14c} (which was posted after the initial version \cite{Matzri14} of the current work)  Min\'a\v c and T\^an also give a Galois-cohomological proof of the Main Theorem, which is similar at several points to our proof;
See also \cite{MinacTan15}.
Moreover, they point out that the standard restriction-correstriction argument allows one to remove the assumption that the field contains a root of unity of order $p$.
Namely, for a $p$th root of unity $\zeta$, the index of $U=G_{F(\zeta)}$ in $G=G_F$ is prime to $p$.
If $\chi_1,\chi_2,\chi_3\in H^1(G)$ and  $\alp\in \langle \chi_1,\chi_2,\chi_3\rangle$, then by our Main Theorem,
$\Res_U\alp=\Res_U(\chi_1)\cup\psi_1+\Res_U(\chi_3)\cup\psi_3$ for some $\psi_1,\psi_3\in H^1(U)$.
Hence $(G\colon U)\alp=\chi_1\cup\Cor_G(\psi_1)+\chi_3\cup\Cor_G(\psi_3)$, and consequently $0\in \langle \chi_1,\chi_2,\chi_3\rangle$ (see \S1).

\bigskip


\section{Massey products}
\label{section on Massey products}
We recall the definition and basic properties of Massey products of degree $1$ cohomology elements.
We first recall that a \textbf{differential graded algebra}  over a ring $R$ (abbreviated $R$-DGA)  is a graded $R$-algebra
$C^\bullet=\bigoplus_{r=0}^\infty C^r$ equipped with $R$-module homomorphisms
$\partial^s\colon C^r\to C^{r+1}$ such that $\partial=\bigoplus_{r=0}^\infty\partial^r$ satisfies $\partial\circ\partial=0$ and one has:
$\partial^{r+s}(ab)=\partial^r(a)b+(-1)^ra\partial^s(b)$ for  $a\in C^r$, $b\in C^s$ (the \textsl{Leibnitz rule}).
Set $Z^r=\Ker(\partial^r)$, $B^r=\Img(\partial^{r-1})$, and $H^r=Z^r/B^r$, and let $[c]$ denote the class of $c\in Z^r$ in $H^r$.
Then $H^\bullet=\bigoplus_{r=0}^\infty H^r$ has an induced $R$-DGA structure with zero differentials $\partial^r$.
We say that the DGA $C^\bullet$ is \textbf{graded-commutative} if $ab=(-1)^{rs}ba$ for $a\in C^r$ and $b\in C^s$.

We fix  an integer $n\geq 2$.
Consider a system $c_{ij}\in C^1$, where $1\leq i\leq j\leq n$ and $(i,j)\neq(1,n)$.
For any $i,j$ satisfying $1\leq i\leq j\leq n$ (including $(i,j)=(1,n)$) we define
\[
\widetilde{c_{ij}}=\sum_{r=i}^{j-1}c_{ir}c_{r+1,j}\in C^2.
\]
One says that $(c_{ij})$ is a \textbf{defining system of size $n$} in $C^\bullet$ if
$\partial c_{ij}=\widetilde{c_{ij}}$ for every $1\leq i\leq j\leq n$ with $(i,j)\neq(1,n)$.
We also say that the defining system $(c_{ij})$ is \textbf{on $c_{11}\nek c_{nn}$}.
Note that then $c_{ii}$ is a $1$-cocycle, $i=1,2\nek n$.
Further, $\widetilde{c_{1n}}$ is a $2$-cocycle (\cite{Kraines66}*{p.\ 432}, \cite{Fenn83}*{p.\ 233}).
Its cohomology class depends only on the cohomology classes $[c_{11}]\nek[c_{nn}]$ \cite{Kraines66}*{Th.\ 3}.
Given $c_1\nek c_n\in Z^1$, the \textbf{$n$-fold Massey product} of $\langle [c_1]\nek [c_n]\rangle$ is the subset of $H^2$ consisting of all cohomology classes $[\widetilde{c_{1n}}]$
obtained from defining systems $(c_{ij})$ of size $n$ on $c_1\nek c_n$ in $C^\bullet$.
The Massey product $\langle [c_1]\nek [c_n]\rangle$ is \textbf{essential} if it is non-empty but does not contain $0$.

When $n=2$, $\langle [c_1],[c_2]\rangle$ is always non-empty and consists only of $[c_1][c_2]$.
In the case $n=3$ one has the following well-known facts:

\begin{prop}[\cite{EfratMatzri15}*{Prop.\ 6.1}]
\label{structure of triple Massey products}
Let $c_1,c_2,c_3\in Z^1$.
\begin{enumerate}
\item[(a)]
$\langle [c_1],[c_2],[c_3]\rangle$ is non-empty if and only if $[c_1][c_2]=[c_2][c_3]=0$;
\item[(b)]
If $(c_{ij})$ is a defining system on $[c_1],[c_2],[c_3]$, then
$\langle [c_1],[c_2],[c_3]\rangle=[\widetilde {c_{13}}]+[c_1]H^1+H^1[c_3]$.
\end{enumerate}
\end{prop}


\section{Cohomological Preliminaries}
We refer, e.g., to \cite{NeukirchSchmidtWingberg} for the basic notions and facts in profinite and Galois cohomology.
Let $p$ be a fixed prime number and let $G$ be a profinite group acting trivially on $\dbZ/p$.
We write $C^r(G)$ for the group $C^r(G,\dbZ/p)$ of continuous (inhomogenous) cochains $G^r\to\dbZ/p$.
Let $Z^r(G)=Z^r(G,\dbZ/p)$ and $B^r(G)=B^r(G,\dbZ/p)$ be its subgroups of $r$-cocycles and $r$-coboundaries, respectively, and let $H^r(G)=H^r(G,\dbZ/p)$ be the corresponding profinite cohomology group.
We identify $H^1(G)=\Hom(G,\dbZ/p)$.
Then $C^\bullet(G)=\bigoplus_{r=0}^\infty C^r(G)$ is a DGA over $\dbF_p$ with the cup product $\cup$.
 Its cohomology DGA $H^\bullet(G)=\bigoplus_{r=0}^\infty H^r(G)$ is graded-commutative.
We will need the following slightly refined version of this property for degree $1$ elements:

\begin{lem}
\label{graded-commutativity}\rm
Let $\chi_1,\chi_2\in H^1(G)$.
Then there exists $\psi\in C^1(G)$ such that $\partial\psi=\chi_1\cup\chi_2+\chi_2\cup\chi_1$ and $\psi$ is zero on $\Ker(\chi_i)$, $i=1,2$.
\end{lem}
\begin{proof}
When $\chi_1,\chi_2$ are $\dbF_p$-linearly independent,
let $\bar G=G/(\Ker(\chi_1)\cap\Ker(\chi_2))\isom(\dbZ/p)^2$, and choose $\bar\sig_1,\bar\sig_2\in \bar G$ which are dual to $\chi_1,\chi_2$.
Define $\bar\psi\in C^1(\bar G)$ by $\bar\psi(\bar\sig_1^i\bar\sig_2^j)=-ij$ for $0\leq i,j<p$, and take $\psi=\Inf_G\bar\psi$ be its inflation to $H^1(G)$.

When $\chi_1,\chi_2$ are nonzero and $\dbF_p$-linearly dependent, we write $\chi_2=k\chi_1$ with $1\leq k<p$ and $\bar G=G/\Ker(\chi_1)\isom\dbZ/p$.
We define $\bar\psi\in C^1(\bar G)$ by $\bar\psi(\bar\sig_1^i)=-ki^2\in\dbZ/p$, and take $\psi=\inf_G\bar\psi$.

Finally, when at least one of  $\chi_1,\chi_2$ is $0$ we take $\psi=0\in C^1(G)$.
\end{proof}

Given a closed subgroup $U$ of $G$ let $\Res_U\colon H^i(G)\to H^i(U)$ be the restriction homomorphism.
When $U$ is open in $G$ we have a correstriction homomorphism $\Cor_G\colon H^i(U)\to H^i(G)$.
If $N$ is a closed normal subgroup of $G$, then every $\sig\in G$ induces a homomorphism $\sig\colon H^1(N)\to H^1(N)$, $\varphi\mapsto\sig\varphi$,
where $(\sig\varphi)(\tau)=\varphi(\sig\tau\sig\inv)$.

For a closed subgroup $U$ of $G$ and for $\chi\in H^1(U)$, we consider the sequence:
\begin{equation}
\label{ex seq}
H^1(\Ker(\chi))\xrightarrow{\Cor_U}H^1(U)\xrightarrow{\chi\cup}H^2(U)\xrightarrow{\Res_{\Ker(\chi)}}H^2(\Ker(\chi)).
\end{equation}

\begin{exam}
\label{exact seq for Galois groups}
\rm
When $G=G_F$ for a field $F$ containing a root of unity of order $p$, this sequence is exact for every such $U$ and $\chi$.
This corresponds to the isomorphism $K^\times/N_{L/K}(L^\times)\isom\Br(L/K)$ for the fixed fields $K,L$ of $U,\Ker(\chi)$, respectively, where $\Br(L/K)$ is the relative Brauer group of the field extension $L\supseteq K$ \cite{Draxl83}*{p.\ 73, Th.\ 1}.
\end{exam}

\begin{prop}
\label{reversing the order}
Suppose that  (\ref{ex seq}) with $U=G$ is exact at $H^2(G)$ for every $\chi\in H^1(G)$.
For every $\chi_1,\chi_2,\chi_3\in H^1(G)$ one has $\langle\chi_1,\chi_2,\chi_3\rangle=\langle\chi_3,\chi_2,\chi_1\rangle$.
\end{prop}
\begin{proof}
Since both Massey products are cosets of $\chi_1\cup H^1(G)+\chi_3\cup H^1(G)$
(Proposition \ref{structure of triple Massey products}(b)), it suffices to show that $\langle\chi_1,\chi_2,\chi_3\rangle\supseteq \langle\chi_3,\chi_2,\chi_1\rangle$.
So let $\alp\in\langle\chi_3,\chi_2,\chi_1\rangle$.
Then there exist $\varphi_{32},\varphi_{21}\in C^1(G)$ such that
\[
\partial\varphi_{32}=\chi_3\cup\chi_2, \quad \partial\varphi_{21}=\chi_2\cup\chi_1, \quad
\alp=[\chi_3\cup\varphi_{21}+\varphi_{32}\cup\chi_1].
\]
Let $K=\Ker(\chi_1)$.
Lemma \ref{graded-commutativity} yields $\psi_{12}\in C^1(G)$ such that
$\partial\psi_{12}=\chi_1\cup\chi_2+\chi_2\cup\chi_1$ in $C^2(G)$ and $\psi_{12}=0$ on $K=\Ker(\chi_1)$.
The graded-commutativity of $H^\bullet(G)$ yields $\psi_{23}\in C^1(G)$ such that
$\partial\psi_{23}=\chi_2\cup\chi_3+\chi_3\cup\chi_2$ in $C^2(G)$.
Taking $\varphi_{12}=\psi_{12}-\varphi_{21}$ and $\varphi_{23}=\psi_{23}-\varphi_{32}$, we obtain that
$\partial\varphi_{12}=\chi_1\cup\chi_2$ and $\partial\varphi_{23}=\chi_2\cup\chi_3$.
It therefore suffices to show that $[\chi_1\cup\varphi_{23}+\varphi_{12}\cup\chi_3]$ and $\alp$ are equal modulo the indeterminicity ${\chi_1\cup H^1(G)+\chi_3\cup H^1(G)}$ of both Massey products.

Now $\Res_K(\partial\varphi_{21})=\Res_K(\chi_2\cup\chi_1)=0$, so $\Res_K\varphi_{21}\in Z^1(K)$.
The graded-commutativity of $H^\bullet(K)$ gives $\Res_K(\varphi_{21}\cup\chi_3+\chi_3\cup\varphi_{21})\in B^2(K)$.
As $\Res_K\psi_{12}=0$ we obtain that
\[
\begin{split}
&\Res_K(\chi_1\cup\varphi_{23}+\varphi_{12}\cup\chi_3)=\Res_K(\varphi_{12}\cup\chi_3)=
-\Res_K(\varphi_{21}\cup\chi_3) \\
\equiv&\Res_K(\chi_3\cup\varphi_{21})=\Res_K(\chi_3\cup\varphi_{21}+\varphi_{32}\cup\chi_1) \pmod{B^2(K)}.
\end{split}
\]
Hence $\Res_K[\chi_1\cup\varphi_{23}+\varphi_{12}\cup\chi_3]=\Res_K\alp$.
By (\ref{ex seq}),
 \[
\alp- [\chi_1\cup\varphi_{23}+\varphi_{12}\cup\chi_3]\in \chi_1\cup H^1(G),
\]
as desired.
\end{proof}

\begin{rem}
\rm
Vogel \cite{VogelThesis}*{Example 1.2.11} proves the assertion of Proposition \ref{reversing the order} under the assumption that $G=F/R$ for a free pro-$p$ group $F$ and a closed normal subgroup $R$ of $F$ contained in the third term of its lower central sequence.
In a topological context, Kraines \cite{Kraines66}*{Th.\ 8} proves  that Massey products of arbitrary length remain the same up to a sign when the order of the entries is reversed.
\end{rem}

\begin{prop}
\label{reductions}
Suppose that  (\ref{ex seq}) with $U=G$ is exact at $H^2(G)$ for every $\chi\in H^1(G)$.
The following conditions are equivalent:
\begin{enumerate}
\item[(1)]
For every $\chi_1,\chi_2,\chi_3\in H^1(G)$, the Massey product $\langle\chi_1,\chi_2,\chi_3\rangle$ is not essential.
\item[(2)]
For every $\chi_1,\chi_2,\chi_3\in H^1(G)$ such that the pairs $\chi_1,\chi_3$ and $\chi_2,\chi_3$ are $\dbF_p$-linearly independent,  $\langle\chi_1,\chi_2,\chi_3\rangle$ is not essential.
\end{enumerate}
\end{prop}
\begin{proof}
(1)$\Rightarrow$(2): \quad Trivial.

\medskip

(2)$\Rightarrow$(1): \quad
Suppose that $\langle\chi_1,\chi_2,\chi_3\rangle\neq\emptyset$.
By Proposition \ref{structure of triple Massey products}(a),  $\chi_1\cup\chi_2=0=\chi_2\cup\chi_3$ in $H^2(G)$.
Therefore there exist $\varphi_{12},\varphi_{23}\in C^1(G)$ such that
$\partial\varphi_{12}=\chi_1\cup\chi_2$ and $\partial \varphi_{23}=\chi_2\cup\chi_3$ in $C^2(G)$.
Then $\chi_1\cup\varphi_{23}+\varphi_{12}\cup\chi_3\in Z^2(G)$.
By Proposition \ref{structure of triple Massey products}(b), we need to find $\varphi_{12},\varphi_{23}$ such that the cohomology class of this $2$-cocycle is contained in
the subset $\chi_1\cup H^1(G)+\chi_3\cup H^1(G)$ of $H^2(G)$.
We break the discussion into several cases.

\medskip

\noindent
Case I: {\sl The pairs $\chi_1,\chi_3$ and $\chi_2,\chi_3$ are $\dbF_p$-linearly independent.} \quad
Then we simply apply (2).

\medskip

\noindent
Case II: {\sl $\chi_1,\chi_3$ are $\dbF_p$-linearly dependent}.  \quad
We may assume that $\chi_1=i\chi_3$ for some $i\in \dbF_p$.
Given $\varphi_{12},\varphi_{23}$ as above we then have
\[
\Res_{\Ker(\chi_3)}(\chi_1\cup\varphi_{23}+\varphi_{12}\cup\chi_3)=0.
\]
By (\ref{ex seq}), $[\chi_1\cup\varphi_{23}+\varphi_{12}\cup\chi_3]\in \chi_3\cup H^1(G)$, and we are done.

\medskip

\noindent
Case III: {\sl $\chi_2=0$.} \quad
Then $\chi_1\cup\chi_2=0=\chi_2\cup\chi_3$ in $C^2(G)$, so for $\varphi_{12}=\varphi_{23}=0$ we have $[\chi_1\cup\varphi_{23}+\varphi_{12}\cup\chi_3]=0$.

\medskip

\noindent
Case IV: {\sl $\chi_1,\chi_3$ are $\dbF_p$-linearly independent, $\chi_2\neq0$, and $\chi_2,\chi_3$ are $\dbF_p$-linearly dependent.} \quad
Then $\chi_1,\chi_2$ are also $\dbF_p$-independent.
By Proposition \ref{reversing the order}, $\langle\chi_1,\chi_2,\chi_3\rangle=\langle\chi_3,\chi_2,\chi_1\rangle$, and by (2), $\langle\chi_3,\chi_2,\chi_1\rangle$ is not essential.
\end{proof}



\section{cup products as coboundaries}
Let $G$ be a profinite group and let $\chi_a,\chi_b\in H^1(G)$ be $\dbF_p$-linearly independent.
Set $N_a=\Ker(\chi_a)$, $N_b=\Ker(\chi_b)$ and $L=N_a\cap N_b$.
Thus $G/L\isom(G/N_a)\times(G/N_b)\isom(\dbZ/p)^2$.
Let $\sig_a,\sig_b\in G$ be dual to $\chi_1,\chi_b$, respectively, i.e.,
\[
\chi_a(\sig_a)=1,\quad \chi_a(\sig_b)=0,\quad\chi_b(\sig_a)=0,\quad \chi_b(\sig_b)=1.
\]
Let $\tau=[\sig_a,\sig_b]=\sig_a\sig_b\sig_a\inv\sig_b\inv$.

\begin{prop}
\label{Ker omega as Hp3}
Suppose that $\omega\in H^1(N_b)$ satisfies $\omega-\sig_b\omega =\Res_{N_b}\chi_a$.
Then
\begin{enumerate}
\item[(a)]
$\omega(\tau)=1$;
\item[(b)]
$N_a\cap\Ker(\omega)$ is normal in $G$;
\item[(c)]
$(G:N_a\cap\Ker(\omega))=p^3$;
\item[(d)]
The images $\bar\sig_a,\bar\sig_b,\bar\tau$ of $\sig_a,\sig_b,\tau$, respectively, in $\bar G=G/(N_a\cap\Ker(\omega))$
generate $\bar G$ and satisfy $[\bar\tau,\bar\sig_a]=[\bar\tau,\bar\sig_b]=1$.
\end{enumerate}
\end{prop}
\begin{proof}
(a)\quad
Since  $\sig_a,\sig_b\sig_a\sig_b\inv\in N_b$, the assumption on $\omega$ gives
\[
\omega(\tau)=\omega(\sig_a)+\omega(\sig_b\sig_a\inv\sig_b\inv)=
\omega(\sig_a)-(\sig_b\omega)(\sig_a)=(\Res_{N_b}\chi_a)(\sig_a)=1.
\]

(b) \quad
For every  $\sig\in N_b$ we have $\sig\omega=\omega$, and therefore $\sig(\Res_L\omega)=\Res_L\omega$.
By the assumption on $\omega$, $\Res_L\omega-\sig_b(\Res_L\omega)=\Res_L\chi_a=0$.
Therefore $\sig(\Res_L\omega)=\Res_L\omega$ for every $\sig\in \langle N_b,\sig_b\rangle=G$.
This means that $\omega(\sig h\sig\inv)=\omega(h)$ for every $\sig\in G$ and $h\in L$.
Consequently, $\Ker(\Res_L\omega)$ is normal in $G$, and we observe that $N_a\cap\Ker(\omega)=\Ker(\Res_L\omega)$.

\medskip

(c) \quad
We note that every commutator in $G$ is contained in $L$.
By this and (a), $\tau\in L\setminus\Ker(\Res_L\omega)$, whence $(L:\Ker(\Res_L\omega))=p$.
Consequently,
\[
(G:N_a\cap\Ker(\omega))=(G:L)(L:\Ker(\Res_L\omega))=p^2\cdot p=p^3.
\]

(d)\quad
The images of $\bar\sig_a,\bar\sig_b$ generate $G/L\isom(\dbZ/p)^2$.
Also, $L/(N_a\cap\Ker(\omega))=L/\Ker(\Res_L(\omega))$ is generated by $\bar\tau$, by (a).
Hence $\bar\sig_a,\bar\sig_b,\bar\tau$ generate $\bar G$.
Since $\sig_a,\tau\in N_b$  we have $\omega(\tau\sig_a\tau\inv\sig_a\inv)=0$, so $\tau\sig_a\tau\inv\sig_a\inv\in N_a\cap\Ker(\omega)$.
Therefore $[\bar\tau,\bar\sig_a]=1$.

As $\tau\in N_a\cap N_b$,
\[
\omega(\tau\sig_b\tau\inv\sig_b\inv)=\omega(\tau)+(\sig_b\omega)(\tau\inv)=\omega(\tau)-(\sig_b\omega)(\tau)
=(\Res_{N_b}\chi_a)(\tau)=0.
\]
Therefore $\tau\sig_b\tau\inv\sig_b\inv\in N_a\cap\Ker(\omega)$, i.e., $[\bar\tau,\bar\sig_b]=1$.
\end{proof}

It follows from Proposition \ref{Ker omega as Hp3} that $\bar G$
is the Heisenberg group $H_{p^3}$ ($D_4$ when $p=2$).
We refer to \cite{SharifiThesis}*{Ch.\ II} for related results.

\begin{prop}
\label{the cochain varphi}
Suppose that $\omega\in H^1(N_b)$ satisfies $\omega-\sig_b\omega =\Res_{N_b}\chi_a$.
There exists $\varphi\in C^1(G)$ with
$\partial\varphi=-\chi_a\cup\chi_b$ in $C^2(G)$ and $\omega=\Res_{N_b}\varphi$ in $C^1(N_b)$.
\end{prop}
\begin{proof}
Let $\bar\chi_a,\bar\chi_b\in Z^1(\bar G)$ be the characters with inflations $\chi_a,\chi_b$, respectively, to $G$.
Every element of $\bar G$ can be uniquely written as $\bar\sig_b^i\bar\sig_a^j\bar\tau^k$ for integers $0\leq i,j,k<p$
(which we also consider as elements of $\dbZ/p$).
We define $\bar\varphi\in C^1(\bar G)$ by $\varphi(\bar\sig)=\omega(\sig_a)j+k$.
Let $\varphi\in C^1(G)$ be the inflation of $\bar\varphi$ to $G$.

To compute $\partial\varphi$, we take $0\leq i,j,k,r,s,t<p$.
Then $\bar\sig_a^j\bar\sig_b^r=\bar\sig_b^r\bar\sig_a^j\bar\tau^{jr}$, so
\[
\bar\varphi(\bar\sig_b^i\bar\sig_a^j\bar\tau^k\bar\sig_b^r\bar\sig_a^s\bar\tau^t)
=\bar\varphi(\bar\sig_b^{i+r}\bar\sig_a^{j+s}\bar\tau^{k+t+jr})
=\omega(\sig_a)(j+s)+k+t+jr.
\]
Therefore
\[
\begin{split}
(\partial\bar\varphi)&(\bar\sig_b^i\bar\sig_a^j\bar\tau^k,\bar\sig_b^r\bar\sig_a^s\bar\tau^t)
=\bar\varphi(\bar\sig_b^i\bar\sig_a^j\bar\tau^k)+\bar\varphi(\bar\sig_b^r\bar\sig_a^s\bar\tau^t)-
\bar\varphi(\bar\sig_b^i\bar\sig_a^j\bar\tau^k\bar\sig_b^r\bar\sig_a^s\bar\tau^t)\\
&=\omega(\sig_a)j+k+\omega(\sig_a)s+t-(\omega(\sig_a)(j+s)+k+t+jr)=-jr\\
&=-\bar\chi_a(\bar\sig_b^i\bar\sig_a^j\bar\tau^k)\bar\chi_b(\bar\sig_b^r\bar\sig_a^s\bar\tau^t)
=-(\bar\chi_a\cup\bar\chi_b)(\bar\sig_b^i\bar\sig_a^j\bar\tau^k,\bar\sig_b^r\bar\sig_a^s\bar\tau^t).
\end{split}
\]
The first equality of the Proposition now follows by inflation to $G$.

For the second equality, let $\sig\in N_b$ and let $\bar\sig$ be the image of $\sig$ in $N_b/(N_a\cap\Ker(\omega))$.
We may write $\bar\sig=\bar\sig_a^j\bar\tau^k$ for some integers $0\leq j,k<p$.
Since $\omega(\tau)=1$ (Proposition \ref{Ker omega as Hp3}(a)) we have
\[
\omega(\sig)=\omega(\sig_a^j\tau^k)=\omega(\sig_a)j+k=\varphi(\sig).
\qedhere
\]
\end{proof}


\section{Massey products containing $0$}
Let $\chi_1,\chi_2,\chi_3\in H^1(G)$,  and set $N_1=\Ker(\chi_1)$,  $N_3=\Ker(\chi_3)$ and $M=N_1\cap N_3$.
Suppose that $\sig_3\in G$ satisfies $\chi_1(\sig_3)=0$ and $\chi_3(\sig_3)=1$.
Also let $\omega\in H^1(N_3)$.
We assume that
\begin{equation}
\label{setup}
\omega-\sig_3\omega=\Res_{N_3}\chi_2, \quad
\chi_1\cup\chi_2=0,
\end{equation}
and $\chi_2,\chi_3$ are $\dbF_p$-linearly independent.

\begin{lem}
\label{res Massey2}
The triple Massey product $\langle \chi_1,\chi_2,\chi_3\rangle$ has a representative $\alp$ such that $\Res_{N_3}\alp=-\Res_{N_3}(\chi_1)\cup \omega$.
\end{lem}
\begin{proof}
Since $\chi_1\cup\chi_2=0$ in $H^2(G)$ there exists $\varphi_{12}\in C^2(G)$ such that $\partial\varphi_{12}=\chi_1\cup\chi_2$ in $C^2(G)$.
Proposition \ref{the cochain varphi} and (\ref{setup}) give rise to $\varphi_{23}\in C^1(G)$ with
$\partial\varphi_{23}=-\chi_2\cup\chi_3$ and $\omega=\Res_{N_3}\varphi_{23}$.
Then $\chi_1\cup(-\varphi_{23})+\varphi_{12}\cup\chi_3$ is a $2$-cocycle with cohomology class $\alp$ in $\langle \chi_1,\chi_2,\chi_3\rangle$.
We have
\[
\Res_{N_3}(\chi_1\cup(-\varphi_{23})+\varphi_{12}\cup\chi_3) =-\Res_{N_3}(\chi_1)\cup \omega,
\]
in $C^2(N_3)$, whence $\Res_{N_3}\alp=-\Res_{N_3}(\chi_1)\cup \omega$ in $H^2(N_3)$.
\end{proof}

\begin{thm}
\label{criterion for trivial Massey product}
In the above setup (\ref{setup}), assume further  that the sequence (\ref{ex seq}) is exact for every open subgroup $U$ of $G$ of index dividing $p$ and every $\chi\in H^1(U)$.
Then the following conditions are equivalent:
\begin{enumerate}
\item[(1)]
$0\in\langle\chi_1,\chi_2,\chi_3\rangle$;
\item[(2)]
There exists $\lam\in H^1(G)$ such that $\Res_{N_3}(\chi_1\cup\lam)=\Res_{N_3}(\chi_1)\cup\omega$;
\item[(3)]
$\omega\in\Res_{N_3} H^1(G)+\Cor_{N_3} H^1(M)$.
\end{enumerate}
\end{thm}
\begin{proof}
(1)$\Rightarrow$(2): \quad
Lemma \ref{res Massey2} yields $\alp\in\langle\chi_1,\chi_2,\chi_3\rangle$ with $\Res_{N_3}\alp=-\Res_{N_3}(\chi_1)\cup\omega$.
Since also $0\in\langle\chi_1,\chi_2,\chi_3\rangle$, Proposition \ref{structure of triple Massey products}(b) gives $\lam,\lam'\in H^1(G)$ such that $-\alp=\chi_1\cup\lam+\chi_3\cup\lam'$.
Now this implies that $\Res_{N_3}\alp=-\Res_{N_3}(\chi_1\cup\lam)$, whence (2).

\medskip

(2)$\Rightarrow$(1): \quad
For $\alp$ as in Lemma \ref{res Massey2}, $\Res_{N_3}(\alp+\chi_1\cup\lam)=0$.
By the exact sequence (\ref{ex seq}),
$\alp+\chi_1\cup\lam\in\chi_3\cup H^1(G)$, whence (1).

\medskip

(2)$\Leftrightarrow$(3): \quad
This follows again from (\ref{ex seq}).
\end{proof}

\section{Kummer formations}
\label{section on Kummer formations}
Let $A$ be a discrete $G$-module.
For a closed normal subgroup $U$ of $G$ let $A^U$ be the submodule of $A$ fixed by $U$.
There is an induced $G/U$-action on $A^U$.

For every open normal subgroups $U\leq U'$ of $G$ let $N_{U'/U}\colon A^U\to A^{U'}$ be the trace map $a\mapsto \sum_\sig\sig a$, where $\sig$ ranges over a system of representatives for the cosets of $U'$ modulo $U$.

Let $I_{U'/U}$ be the subgroup of $A^U$ consisting of all elements of the form $\bar\sig a-a$ with $\bar\sig\in U'/U$ and $a\in A$.
We recall that
\[
\hat H^{-1}(U'/U,A^U)=\Ker(N_{U'/U})/I_{U'/U}.
\]
When $U'/U$ is cyclic with generator $\bar\sig$, the subgroup $I_{U'/U}$ consists of all elements $\bar\sig a-a$, with $a\in A$ (since $\bar\sig^k-1=(\bar\sig-1)\sum_{i=0}^{k-1}\bar\sig^i$).
Then $\hat H^{-1}(U'/U,A^U)\isom H^1(U'/U,A^U)$  \cite{NeukirchSchmidtWingberg}*{Prop.\ 1.7.1}.

\begin{defin}
\rm
A \textbf{$p$-Kummer formation} $(G,A,\{\kappa_U\}_U)$ consists of a profinite group $G$, a discrete $G$-module $A$, and for each open normal subgroup $U$ of $G$ a $G$-equivariant epimorphism $\kappa_U\colon A^U\to H^1(U)$ such that for every open normal subgroup $U$ of $G$ the following conditions hold:
\begin{enumerate}
\item[(i)]
the sequence (\ref{ex seq}) is exact for every $\chi\in H^1(U)$;
\item[(ii)]
$\Ker(\kappa_U)=pA^U$;
\item[(iii)]
for every open normal subgroup $U'$ of $G$ such that $U\leq U'$, there are commutative squares
\[
\xymatrix{
A^U\ar[r]^{\kappa_U} & H^1(U)&&A^U\ar[r]^{\kappa_U}\ar[d]_{N_{U'/U}} & H^1(U)\ar[d]^{\Cor_{U'}} \\
A^{U'}\ar@{^{(}->}[u]\ar[r]^{\kappa_{U'}} & H^1(U')\ar[u]_{\Res_U};&&A^{U'}\ar[r]^{\kappa_{U'}} & H^1(U');
}
\]
\item[(iv)]
for every open normal subgroup $U'$ of $G$ such that $U\leq U'$ and $(U':U)=p$ one has $\hat H^{-1}(U'/U,A^U)=0$.
\end{enumerate}
\end{defin}

\begin{exam}
\label{Kummer formations of fields}
\rm
Let $F$ be a field which contains a root of unity of order $p$.
We fix an isomorphism between the group $\mu_p$ of $p$th roots of unity and $\dbZ/p$.
Given an open subgroup $U$ of $G_F$ let $E=F_\sep^U$ be its fixed field.
The \textbf{Kummer homomorphism} $\kappa_U\colon E^\times\to H^1(U)$ is the connecting homomorphism arising from the short exact sequence of $U$-modules
\[
0\to\dbZ/p\to F_{\rm sep}^\times\xrightarrow{p}F_{\rm sep}^\times\to 1.
\]
By Hilbert's Theorem 90 it is surjective.
Then $(G_F,F_{\rm sep}^\times,\{\kappa_U\}_U)$ is a $p$-Kummer formation.
Indeed, (i) was pointed out in Example \ref{exact seq for Galois groups}.
(ii) is the standard fact that $\Ker(\kappa_U)=(E^\times)^p$, and (iii)  follows from the commutativity of connecting homomorphisms with restrictions and correstrictions.
For (iv) use the isomorphism  $\hat H^{-1}(U'/U,A^U)=H^1(U'/U,A^U)$ for $U'/U$ cyclic and Hilbert's Theorem 90.
\end{exam}

\begin{prop}
\label{key prop}
Let $(G,A,\{\kappa_U\}_U)$ be a $p$-Kummer formation.
Let $M_1,M_3$ be distinct normal subgroups of $G$ of index $p$, let $M=M_1\cap M_3$, and
let  $\sig_3\in M_1$ satisfy $G=\langle M_3,\sig_3\rangle$.
Suppose that $\lam_1\in H^1(M_1)$ and $\lam_3\in H^1(M_3)$ satisfy $\Cor_G\lam_1=\Cor_G\lam_3$.
Then there exists  $\omega\in H^1(M_3)$ such that
\[
 \sig_3\omega-\omega=-\Res_{M_3}\Cor_G\lam_3, \quad \omega\in\Res_{M_3}H^1(G)+\Cor_{M_3}H^1(M).
\]
\end{prop}
\begin{proof}
There exist $y_1\in A^{M_1}$ and $y_3\in A^{M_3}$ such that $\kappa_{M_1}(y_1)=\lam_1$ and $\kappa_{M_3}(y_3)=\lam_3$.
Let $w=\sum_{i=0}^{p-1}i\sig_3^iy_3$, and note that $w\in A^{M_3}$.
We have $(\sig_3-1)\sum_{i=0}^{p-1}i\sig_3^i=(p-1)\sig_3^p+1-\sum_{i=0}^{p-1}\sig_3^i$ in $\dbZ[G]$.
As $\sig_3^p\in M_3$ this gives
\[
(\sig_3-1)w=\bigl((p-1)\sig_3^p+1-N_{G/M_3}\bigr)y_3=py_3-N_{G/M_3}y_3.
\]
Setting $\omega=\kappa_{M_3}(w)\in H^1(M_3)$, the $G$-equivariance of $\kappa_{M_3}$ and assumption (iii) imply that
\[
\begin{split}
\sig_3\omega-\omega&=\kappa_{M_3}((\sig_3-1)w)=-\kappa_{M_3}(N_{G/M_3}y_3)=-\Res_{M_3}\kappa_G(N_{G/M_3}y_3)\\
&=-\Res_{M_3}\Cor_G\kappa_{M_3}(y_3)=-\Res_{M_3}\Cor_G\lam_3.
\end{split}
\]

By (iii),
\[
\begin{split}
\kappa_G(N_{G/M_1}y_1-N_{G/M_3}y_3)&=\Cor_G\kappa_{M_1}(y_1)-\Cor_G\kappa_{M_3}(y_3)\\
&=\Cor_G\lam_1-\Cor_G\lam_3=0.
\end{split}
\]
From (ii) we obtain $b\in A^G$ such that $N_{G/M_1}y_1-N_{G/M_3}y_3=pb$.

Next we choose  $\sig_1\in M_3$ such that $G=\langle M_1,\sig_1\rangle$,
and denote $M'=\langle M,\sig_1\sig_3\rangle$.
We note that $\sig_1,\sig_3$ commute modulo $M$, so $N_{M'/M}=\sum_{i=0}^{p-1}\sig_1^i\sig_3^i$ on $A^M$.
Therefore $N_{M'/M}=N_{G/M_3}$ on $A^{M_3}$, and $N_{M'/M}=N_{G/M_1}$ on $A^{M_1}$.
We obtain that
\[
N_{M'/M}(y_3-y_1+b)=N_{G/M_3}y_3-N_{G/M_1}y_1+pb=0.
\]
By (iv), $\hat H^{-1}(M'/M,A^M)=0$, so $y_3-y_1+b=(\sig_1\sig_3-1)t$ for some $t\in A^M$.
Therefore
\[
\begin{split}
(\sig_3-1)w&=py_3-N_{G/M_3}y_3=N_{M_3/M}y_3-N_{G/M_1}y_1+pb\\
&=N_{M_3/M}y_3-N_{M_3/M}y_1+pb=N_{M_3/M}(y_3-y_1+b)\\
&=N_{M_3/M}(\sig_1\sig_3-1)t=\sig_3\sig_1N_{M_3/M}t-N_{M_3/M}t=(\sig_3-1)N_{M_3/M}t,
\end{split}
\]
since $\sig_1N_{M'/M}=N_{M'/M}$ on $A^M$.
Thus $w-N_{M_3/M}t\in A^{\langle M_3,\sig_3\rangle}=A^G$.
Taking $\eta=\kappa_M(t)\in H^1(M)$, we obtain using (iii) that
\[
\omega-\Cor_{M_3}\eta=\kappa_{M_3}(w-N_{M_3/M}t)=\Res_{M_3}\kappa_G(w-N_{M_3/M}t)\in\Res_{M_3}H^1(G).
\]
Consequently, $\omega\in\Res_{M_3}H^1(G)+\Cor_{M_3}H^1(M)$.
\end{proof}

\begin{thm}
\label{Main Theorem for formations}
Let $(G,A,\{\kappa_U\}_U)$ be a $p$-Kummer formation and let $\chi_1,\chi_2,\chi_3\in H^1(G)$.
Then the Massey product $\langle\chi_1,\chi_2,\chi_3\rangle$ is not essential.
\end{thm}
\begin{proof}
We assume that $\langle\chi_1,\chi_2,\chi_3\rangle$ is non-empty.
By Proposition \ref{structure of triple Massey products}(a),  $\chi_1\cup\chi_2=0=\chi_2\cup\chi_3$.
By Proposition \ref{reductions}, we may assume that the pairs $\chi_1,\chi_3$ and $\chi_2,\chi_3$ are $\dbF_p$-linearly independent.

Let $M_1=\Ker(\chi_1)$, $M_3=\Ker(\chi_3)$, and $M=M_1\cap M_3$, and choose $\sig_3\in M_1$ such that $G=\langle M_3,\sig_3\rangle$.
The exact sequence (\ref{ex seq}) yields $\lam_1\in H^1(M_1)$ and $\lam_3\in H^1(M_3)$ such that $\Cor_G\lam_1=\chi_2=\Cor_G\lam_3$.
Proposition \ref{key prop} gives rise to $\omega\in H^1(M_3)$ such that $\sig_3\omega-\omega=-\Res_{M_3}\chi_2$
and $\omega\in\Res_{M_3}H^1(G)+\Cor_{M_3}H^1(M)$.
By Theorem \ref{criterion for trivial Massey product}, $0\in\langle\chi_1,\chi_2,\chi_3\rangle$.
\end{proof}

\bigskip

Theorem \ref{Main Theorem for formations} and Example \ref{Kummer formations of fields} imply the Main Theorem.

\end{document}